\documentclass[11pt]{article}

\usepackage{amsmath,amssymb,amsthm,amscd,epsfig}
\newcommand{\vp}{\varepsilon}

\theoremstyle{plain}

\newtheorem{thm}{Theorem}

\newtheorem*{unthm}{Theorem}

\newtheorem{pro}{Proposition}

\theoremstyle{definition}

\newtheorem{defn}{Definition}

\theoremstyle{remark}

\begin{document}

\title{Isomorphism classes for certain expanding maps and their group extensions}

\author{Eugen Mihailescu}
\date{}
\maketitle

\begin{abstract}

We prove that expanding toral endomorphisms, together with their respective Lebesgue (Haar) measure are isomorphic to
 1-sided Bernoulli shifts, and are thus as far from invertible as possible, from a measure-theoretic point of view.
  This result is then extended to systems of perturbations of expanding toral endomorphisms, together with their
  respective measures of maximal entropy. Also we study group extensions of expanding toral endomorphisms, in particular
   probabilistic systems on skew products with (other) tori; and prove that under certain, not too restrictive conditions
   on the extension cocycle, these skew products are isomorphic to 1-sided Bernoulli shifts as well.

\end{abstract}



\section{Introduction and outline of main results.}\label{sec1}

From the point of view of ergodic properties, endomorphisms of
Lebesgue spaces behave in a very different way from automorphisms.
The study of measure-preserving endomorphisms presents also many
different methods and ideas than that of
diffeomorphisms/automorphisms (for example \cite{HR}, \cite{L},
\cite{Ma-85}, \cite{PW}, \cite{Pa-96}, \cite{S}, \cite{M-MZ},
\cite{MU-BLMS}, etc.)

Even for 1-sided Bernoulli shifts, the problem of classification up to measure-theoretic isomorphism is very far from the similar problem for 2-sided Bernoulli maps. For 2-sided Bernoulli shifts, it is well known that Ornstein showed they can be classified by measure-theoretic entropy alone (for example \cite{Ma}, \cite{OW}), while for 1-sided shifts this is not at all the case; in fact as Parry and Walters (\cite{PW}, \cite{Wa}) showed,  endomorphisms $T$ on Lebesgue spaces $(X, \mathcal{B}, \mu)$ cannot be classified even by a combination of entropy, Jacobian and the sequence of decreasing algebras $\{T^{-n}\mathcal{B}\}_{n\ge 0}$.

The existence of multiple preimages of a point and the possibly different behavior of consecutive sums on different
 prehistories of points, imply that the dynamical and ergodic properties of endomorphisms are different in results and in
  techniques, than those of automorphisms.
Endomorphisms were studied under various aspects, both from the
point of view of smooth dynamical behavior in the
expanding/hyperbolic case, as well as from the point of view of
ergodic/statistical properties by several authors, for instance
\cite{CP}, \cite{HR}, \cite{K}, \cite{L}, \cite{Ma-85},
\cite{Pa-96}, \cite{PW}, \cite{Ru-exp}, \cite{S}, \cite{M-MZ},
\cite{M-DCDS06}, \cite{MU-BLMS}, etc.

In \cite{K}, Katznelson proved that an automorphism of the $m$-dimensional torus $\mathbb T^m, m \ge 2$ whose eigenvalues
 are never roots of unity, must be a 2-sided Bernoulli shift.
 However in the case of an endomorphism $f_A$ of $\mathbb T^m$ induced by the matrix $A$,  the situation is
 completely different
  and we cannot say a priori whether $f_A$ is 1-sided Bernoulli with respect to the corresponding Lebesgue (Haar)
  measure; we know that its natural extension is 2-sided
  Bernoulli, but this does not tell us much about the possible
  1-sided Bernoullicity of $f_A$.

In order to tackle the problem whether an expanding toral endomorphism $f_A$ (i.e a group  endomorphism on $\mathbb T^m$
given by a matrix $A$, all of whose eigenvalues are strictly larger than 1 in norm) is 1-sided Bernoulli, we need the notion of \textit{tree very weakly Bernoulli}, introduced by Hoffman and Rudolph in \cite{HR}. This condition is similar to that of very weakly Bernoulli automorphism (see \cite{K}, \cite{OW}, \cite{KS}, etc.), and is suited to deal with the whole tree of prehistories for an endomorphism.

In Theorem \ref{main} we will show that an expanding toral endomorphism $f_A$ is a \textit{uniform} constant-to-1 endomorphism with respect to the Lebesgue measure, and that it is tree very weakly Bernoulli;  hence by \cite{HR}, it will follow that it is 1-sided Bernoulli.

We extend this result also to perturbations $g$ of an expanding toral endomorphism $f_A$ on $\mathbb T^m$, together with their respective measure of maximal entropy $\mu_g$ on $\mathbb T^m$. Such systems are shown to be 1-sided Bernoulli too.

We then study examples of group extensions for expanding toral
endomorphisms $f_A$, given by weakly mixing skew products
associated to summable cocycles with values in tori (for instance
\cite{AS}, \cite{CP}, \cite{HR}, \cite{KN}, \cite{Pa-96},
\cite{PP}, etc.) These are maps $f_{A, \psi} : \mathbb T^m \times
\mathbb T^k \to \mathbb T^m \times \mathbb T^k$, of type $$ f_{A,
\psi}(x, y_1, \ldots, y_k) = (f_A(x), \psi_1(x) + y_1 \
(\text{mod} \ 1), \ldots, \psi_k(x) + y_k \ (\text{mod}\ 1)), $$
with $ x \in \mathbb T^m, (y_1, \ldots, y_k) \in \mathbb T^k$, and
where $\psi: \mathbb T^m \to \mathbb T^k$ is a Holder continuous
function.

The conditions for a toral extension $f_{A, \psi}$ to be weakly mixing with respect to the product of Haar measures on $\mathbb T^m$ and $\mathbb T^k$, are related to the "liniar independence" of the components of the cocycle; or in other words to the fact that $\psi$ is not a coboundary (see \cite{AS}, \cite{KN}, \cite{Pa-97}, \cite{PP}, etc.)

We can find in this way a large class of examples of group extensions of expanding toral endomorphisms which are shown to be 1-sided Bernoulli as well.
Such toral extensions can be constructed also for perturbations $g$ of $f_A$, as above.

\section{1-sided Bernoulli toral maps. Group extensions.}

In the sequel we work with \textbf{Lebesgue space systems}, i.e with measurable endomorphisms $f: X \to X$ on Lebesgue spaces $(X, \mathcal{B}, \mu)$ s.t $f$ preserves the probability measure $\mu$. In some cases, when there is no confusion on the $\sigma$-algebra $\mathcal{B}$ on $X$, we shall write only $(X, \mu)$ for the Lebesgue space and $(X, \mu, f)$ for the system.

In particular we shall investigate toral endomorphisms of type $f_A: \mathbb T^m \to \mathbb T^m, m \ge 2$, given by an integer-valued matrix $A$ all of whose eigenvalues are strictly larger than 1 in absolute value. In this case we see that $f_A$ is also a distance expanding map with respect to the Riemannian metric on $\mathbb T^m$.
We consider the Lebesgue (Haar) measure $\mu_m$ on $\mathbb T^m$, which clearly is preserved by $f_A$; in fact $\mu_0$ is the unique measure of maximal entropy for $f_A$.

In \cite{K} Katznelson showed that a group automorphism $T$ of $\mathbb T^m$, none of whose eigenvalues are roots of unity, is in fact a \textit{Bernoulli (2-sided) shift} with respect to $\mu_m$. This means that there exists a measure-theoretic conjugacy (isomorphism) between the  Lebesgue space systems $(\mathbb T^m, \mu_m, T)$ and $(\Sigma_r, \nu_p, \sigma_r)$, where $\Sigma_r$ is the space of all bi-infinite sequences formed on $r$ symbols, $\sigma_r: \Sigma_r \to \Sigma_r$ is the shift map and $\nu_p$ is the probability measure on $\Sigma_r$ induced by the probability vector $p = (p_1, \ldots, p_r), p_1 + \ldots + p_r = 1$, in such a way that $\nu_p(\{\omega \in  \Sigma_r, \omega_0 = i \}) = p_i, i = 1, \ldots, r$.

A \textbf{1-sided Bernoulli shift} is a Lebesgue space system $(X, \mu, f)$ isomorphic to the \textit{model} system $(\Sigma_r^+, \sigma_r, \nu_p)$, where $\Sigma_r^+ = \{(\omega_0, \omega_1, \ldots), \omega_j \in \{1, \ldots, r\}, j \ge 0\}$ is the space of positively indexed sequences on $r$ symbols, and $\sigma_r, \nu_p$ are as before. A \textit{uniform model} 1-sided Bernoulli shift $(\Sigma_r^+, \nu_{(\frac 1r, \ldots, \frac 1r)}, \sigma_r)$ corresponds to the uniformly distributed probability vector $p$  with $p_i = \frac 1r, i = 1, \ldots, r$.

As said in the Introduction, ergodic properties of 1-sided Bernoulli shifts are very different than those for 2-sided Bernoulli shifts, due to their non-invertibility (\cite{PW}, \cite{Wa}, etc.) \
 Hoffman and Rudolph (\cite{HR}) introduced a notion of \textit{tree very weakly Bernoulli} to deal with the whole tree of prehistories that a point may have with regards to an endomorphism. The notion of \textit{very weakly Bernoulli} (for instance \cite{KS}, \cite{K}, etc.) involves automorphisms which interchange the sets in partitions of type $T^k \xi, k \ge 1$ (for a certain partition $\xi$). To define the notion of tree very weakly Bernoulli in the case of endomorphisms, one  uses the (many) partial inverses of the endomorphism $T$ and then automorphisms which preserve the tree structure on the tree of prehistories given by those partial inverses.

Let us remind therefore the definition of \textit{tree very weakly Bernoulli}.

Assume that $f:X \to X$ is a measure-preserving endomorphism (measurable map) on a Lebesgue space $(X, \mathcal{B}, \mu)$.
Now one can take a measurable partition of $X$ by fibers of $f$ of type $f^{-1}(x), x \in X$. Associated to this partition there exists the family of conditional measures of $\mu$, denoted by $\{\mu_x\}_x$; for $\mu$-almost every $x \in X$, the conditional measure $\mu_x$ is a probability measure on the fiber $f^{-1}(x)$.

Assume next that $f$ is \textbf{uniformly $r$-to-one} (\cite{HR}): i.e that it has measure-theoretic entropy $h_\mu(f) = \log r$, that $\mu$-almost all points in $X$ have $r$ $f$-preimages in $X$ and, moreover that the conditional expectations of the preimages are all equal to $\frac {1}{r}$.  The last condition means that the conditional measure $\mu_x$ constructed above on $f^{-1}(x)$, is uniformly distributed over the $r$ preimages of $x$ from $f^{-1}(x)$, for $\mu$-almost all $x \in X$.

Let us now denote by $\mathcal{T}$ an abstract model of infinite
tree having exactly $r^n$ nodes at each index $n \ge 0$. In this
tree, the root node at level 0 is unlabeled; then we label each
node at index 1 by a value in $\{0, \ldots, r-1\}$. Then at index
2 we connect each node from level (index) 1 by one other node
labeled by a value in $\{0, \ldots, r-1\}$. In this way we label a
node $v$ at level $n\ge 1$ by the sequence of values from $\{0,
\ldots, r-1\}$, which sequence joins the root with $v$. This gives
us also a way to concatenate nodes, for instance by $vw$ we mean a
node rooted at $v$ and continued with $w$. If $v$ is a node at
level (index) $n$ of $\mathcal{T}$, then we say also that it has
\textit{length} $n$.

By $\mathcal{T}_n$ we shall denote the finite subtree of $\mathcal{T}$ that has all the nodes of $\mathcal{T}$ at levels smaller than or equal to $n$.

Now for $Y$ a compact metric space, a $(\mathcal{T}, Y)$-\textit{name} is by definition any function
$h: \mathcal{T} \to Y$; the function $h$ is called \textit{tree adapted}
if for any node $v$ and $i, j \in \{0, \ldots, r-1\}, i \ne j$, we
have $h(vi) \ne h(vj)$.

If $\mathcal{B}(Y)$ denotes the borelian $\sigma$-algebra on the compact metric space $Y$ and if $\phi: X \to Y$ is a map, then denote by $\mathcal{F}(\phi)$ the $\sigma$-algebra on $X$ given by the pullback of $\mathcal{B}(Y)$ through $\phi$. We shall say that $\phi$ \textit{generates} with respect to the measure-preserving system $(X, \mathcal{B}, \mu, f)$ if $$\vee_i f^{-i}(\mathcal{F}(\phi)) = \mathcal{B},$$ i.e if by taking the join algebra of the consecutive pullbacks of $\mathcal{F}(\phi)$, we obtain the original $\sigma$-algebra $\mathcal{B}$ on $X$.

Given a uniform $r$-to-1 endomorphism $f$ on the Lebesgue space $(X, \mathcal{B}, \mu)$, a map $\phi:X \to Y$ and a point $x \in X$, we shall
denote by $\mathcal{T}_x^\phi$ the $(\mathcal{T}, Y)$-name
given by the values of $\phi$ on the tree of consecutive $f$-preimages of $x$ from X. Namely first consider a \textit{Rokhlin partition}, i.e a measurable partition of $(X, \mathcal{B}, \mu)$ into sets $\{E_0, \ldots, E_{r-1}\}$ so that $f:E_i \to X$ be bijective $\mu$-almost everywhere, for all $i = 0, \ldots, r-1$. This allows us to define $r$ "local" inverse branches of $f$, i.e we take $f_i:X \to E_i$ to be the inverse of $f|_{E_i}: E_i \to X, i = 0, \ldots, r-1$. And then for a node $v \in \mathcal{T}$ at level $n \ge 1$, obtained from the sequence of symbols $j_1, \ldots, j_n$ from $\{0, \ldots, r-1\}$, with $j_k$ representing the symbol at level $k$, we  define the inverse iterate $f_v(x) := f_{j_n}(\ldots(f_{j_1}(x))\ldots)$.  Now the $(\mathcal{T}, Y)$-\textit{name of $\phi$-values on the preimages of $x$}, denoted by $\mathcal{T}_x^\phi$, is given as $$\mathcal{T}_x^\phi(v):= \phi(f_v(x)), v \in \mathcal{T}$$

Next we say that a map $\phi: X \to Y$ is \textit{tree adapted} on $X$ if the associated name $\mathcal{T}_x^\phi$ of $\phi$-values on the preimage tree of $x$, is a tree adapted $(\mathcal{T}, Y)$-name for $\mu$-almost every $x \in X$.

Let us denote now by $\mathcal{A}$ the collection of
\textit{tree automorphisms} on $\mathcal{T}$, i.e the collection of bijections from the set of nodes of $\mathcal{T}$ to itself, which preserve the
tree structure; i.e $A$ takes a node of type $v = (v_0, \ldots, v_k)$ into a node of type $(A(v_0, \ldots, v_{k-1}), v_k')$. Denote by $\mathcal{A}_n$ the set of bijections on the set of nodes up to level $n$, preserving again the tree
structure. Then for any $n > 1$, we can define a metric on the
space of $(\mathcal{T}, Y)$-names by
$$
 t_n(g, h):= \mathop{\inf}\limits_{A \in \mathcal{A}_n} \frac
1n \mathop{\sum}\limits_{0 < |v| \le n} \frac{1}{r^{|v|}} d(h(v),
g(Av))
$$
Then as in \cite{HR} we define:

\begin{defn}\label{tvwB}
We say that the uniform $r$-to-1 measure preserving system $(X, \mathcal{B}, \mu, f)$ and the tree adapted map $\phi: X \to Y$ are \textbf{tree very weakly Bernoulli} if for any $\vp>0$ and all $n$ large enough, there exists a set $G(\vp, n)$ with $\mu(G(\vp, n)) > 1-\vp$, such that
$t_n(\mathcal{T}_z^\phi, \mathcal{T}_w^\phi) < \vp, \forall z, w \in G(\vp, n)$
\end{defn}

\begin{thm}\label{main}
Let $f_A$ be a toral endomorphism on $\mathbb T^m, m \ge 2$, given by the integer-valued matrix $A$, all of whose eigenvalues are strictly larger than 1 in absolute value.
Then the endomorphism $f_A$ on the torus $\mathbb T^m$ equipped with its Lebesgue (Haar) measure $\mu_m$, is isomorphic to a uniform model 1-sided Bernoulli shift.

\end{thm}

\begin{proof}

From the fact that all the eigenvalues of $A$ are larger than 1 in absolute value, it follows that $f_A$ is an expanding map on $\mathbb T^m$. Also it is well-known that $f_A$ is $|\text{det}(A)|$-to-1 on $\mathbb T^m$ (for instance \cite{Wa}).
Assume $|\text{det}(A)| = r \ge 2$ (otherwise $f_A$ is an automorphism, and toral automorphisms  are 2-sided Bernoulli shifts by \cite{K}).

Now if $f_A$ is expanding and $r$-to-1, we know that $\mu_m$ is
the limit of a sequence of probability measures of type $\nu_n^x:=
\frac{1}{r^n} \mathop{\sum}\limits_{y \in f_A^{-n}x}
\delta_y$, for some $x \in \mathbb T^m$ (see for example \cite{Ru-exp}, \cite{Ma}). This implies that $$\mu_m(f_A(B)) = r \mu_m(B),$$ for any borelian set $B$ so that $f|_B$ is
injective.  Thus the conditional
probabilities of $\mu_m$, associated to the partition $\xi$ into fibers $\{f^{-1}(z)\}_z$, are equidistributed  on the fibers of $\mu_m$-almost all points from
$\mathbb T^m$, i.e such a conditional probability gives weight equal to $\frac 1r$ to each of the $r$ $f_A$-preimages of $z$.

But $\mu_m$ is the Lebesgue (Haar) measure on $\mathbb T^m$ and $f_A$ is supposed to be $r$-to-1, hence the entropy $h_{\mu_m}(f_A) = \log r$.
So $h_{\mu_m}(f_A) = \log r$, the conditional probabilities of $\mu_m$ on the
preimages are all equal to $\frac 1r$ and $f_A$ is $r$-to-1, meaning that $(\mathbb T^m, \mathcal{B}, \mu_m, f_A)$ is a \textit{uniform measure
preserving endomorphism}.

Now $f_A$ is expanding and open (since $f|_\Lambda$ is $r$-to-1), hence $f|_\Lambda$ is topologically exact. Thus for any $\vp>0$ small there exists some positive
integer $N$ (independent of $y, z$) so that, given any $y, z \in \mathbb T^m$ and any $N$-preimage $y_{-N}$ of $y$, there exists an
$N$-preimage $z_{-N}$ of $z$, such that
\begin{equation}\label{preim}
d(y_{-N}, z_{-N}) < \vp
\end{equation}

As our \textit{generating} function we will take the identity
$Id: \mathbb T^m \to \mathbb T^m$ which clearly generates the
$\sigma$-algebra of borelians on $\mathbb T^m$.

From (\ref{preim}), and the fact that local inverse iterates of $f$ contract distances, we infer that given any points $y, z \in \mathbb T^m$, there exists $N = N(\vp)$ such that for any $n> N$ and any $n$-preimage $y_{-n}$ of $y$, there exists a unique  $n$-preimage $z_{-n}$ of $z$, so that $z_{-n} \in B_n(y_{-n}, \vp)$; and vice-versa, for any $n$-preimage $z_{-n} \in \Lambda$ of $z$, there is a unique $n$-preimage $y_{-n} \in \Lambda$ of $y$ with $y_{-n} \in B_n(z_{-n}, \vp)$.
Therefore for any $\vp>0$, there exists $N(\vp)$ so that we have:
\begin{equation}\label{Rudolph}
 t_n(\mathcal{T}_y, \mathcal{T}_z) < C\vp, \forall y, z \in \mathbb T^m, n > N(\vp)
\end{equation}
where $C>0$ is a constant, independent of $\vp, n, y, z$ ($C$ depends only on the minimum expansion coefficient of $f_A$ on $\mathbb T^m$). So in our case the set $G(\vp, n)$ from the definition of tree very weakly Bernoulli, is the whole $\mathbb T^m$.

Thus the measure preserving uniform endomophism $(\mathbb T^m, \mathcal{B},  \mu_m, f_A)$ and the generating function $Id: \mathbb T^m \to \mathbb T^m$ are tree very weakly Bernoulli. In conclusion from \cite{HR} we see that $(\mathbb T^m, \mathcal{B}, \mu_m, f_A)$ is 1-sided Bernoulli and conjugate to the uniform model $(\Sigma_r^+, \nu_{(\frac 1r, \ldots, \frac 1r)}, \sigma_r)$.

\end{proof}

We now consider $\mathcal{C}^2$-smooth \textbf{perturbations} $g: \mathbb T^m \to \mathbb T^m$ of the expanding endomorphism $f_A$, in the $\mathcal{C}^1$ topology. Such a map $g$ has a unique measure of maximal entropy denoted by $\mu_g$, on $\mathbb T^m$. We have that $\mu_g$ is absolutely continuous with respect to the Lebesgue measure $\mu_m$ and that the Radon-Nykodim derivative $\frac{d\mu_g}{d\mu_m}$ is Holder continuous, and bounded away from 0 and $\infty$ (see for instance \cite{Ma}).

\begin{thm}\label{perturb}
Let $A$ be an integer-valued $m\times m$ matrix, all of whose eigenvalues are strictly larger than 1 in absolute value, and let $f_A$ be the associated toral endomorphism on $\mathbb T^m$. Assume that $g$ is a $\mathcal{C}^2$ perturbation of $f_A$ in the $\mathcal{C}^1$-topology; and denote by $\mu_g$ the unique measure of maximal entropy of $g$ on $\mathbb T^m$. Then the system $(\mathbb T^m, \mu_g, g)$  is 1-sided Bernoulli.
\end{thm}

\begin{proof}
First since $g$ is a perturbation of the expanding endomorphism $f_A$ on $\mathbb T^m$, it follows from the results of Shub (\cite{S}) that $g$ is topologically conjugate to $f_A$ on $\mathbb T^m$. Thus there exists a homeomorphism $H: \mathbb T^m \to \mathbb T^m$ so that $$H\circ g = f_A \circ H$$
This implies easily that, since $f_A$ is $r$-to-1 (where $r = |\text{det} A|$), then $g$ is also $r$-to-1 on $\mathbb T^m$.
Also, if $g$ is topologically conjugate to $f_A$, we obtain that the topological entropy of $g$ is the same as the topological entropy of $f_A$, i.e
$$h_{top}(g) = h_{\mu_g}(g) = h_{top}(f_A) = \log r$$

Next if $g$ is sufficiently close to $f_A$ in $\mathcal{C}^1$
topology, it follows that $g$ is expanding too. Hence we can apply
the results of \cite{Ma}, \cite{Ru-exp} in order to obtain the
unique measure of maximal entropy $\mu_g$ of $g$, as the limit of
the sequence of probabilities $$\nu^x_{g, n}:=
\frac{\mathop{\sum}\limits_{y \in g^{-n}(x)}\delta_y}{r^n}, \ n
\ge 1$$ But this shows as in Theorem \ref{main} that $\mu_g(g(B))
= r \mu_g(B)$, for any borelian set $B \subset \mathbb T^m$, thus
the conditional measures of $\mu_g$ associated to the fiber
partition, are equally distributed among the $r$ preimages in
almost all fibers of $g$.

Thus we obtain from the considerations so far that $g$ is a uniform $r$-to-1 endomorphism with respect to the measure $\mu_g$ on $\mathbb T^m$.

Now we proceed as in the proof of Theorem \ref{main} in order to obtain that $g$ is tree very weakly Bernoulli and then from \cite{HR}, the system $(\mathbb T^m, \mu_g, g)$ is 1-sided Bernoulli.

\end{proof}

\textbf{Example 1.}  One example of a perturbation of an expanding
toral endomorphism is given by the map $g:\mathbb T^2 \to \mathbb
T^2$ $$g(x, y) = (2x + \vp \sin (2\pi x + 4\pi y), 3y + \vp
\cos(2\pi x)), (x, y) \in \mathbb T^2$$ The expanding map $g$ has
a unique measure of maximal entropy $\mu_g$ on $\mathbb T^2$, and
this measure is absolutely continuous with respect to the Haar
measure, although it is not necessarily equal to it. We see from
Theorem \ref{perturb} that $g$ is 1-sided Bernoulli with respect
to its measure of maximal entropy $\mu_g$.

\ \

We will study now an interesting class of endomorphisms which can be constructed starting with some known endomorphisms, namely \textbf{group extensions} (in our case with tori). Many important aspects of group extensions given by skew products have been investigated in the literature, for instance in \cite{AS}, \cite{CP}, \cite{HR}, \cite{KN}, \cite{Pa-96}, \cite{Pa-97}, \cite{PP}, \cite{R}, etc.

Let us start with a measure-preserving endomorphism $f$ on a Lebesgue space $f: (X, \mathcal{B}, \mu) \to (X, \mathcal{B}, \mu)$. Consider also a compact metric space $(Z, d)$ with $Isom(Z)$ being the space of its isometries (with uniform topology). Assume that $Isom(Z)$ acts transitively on $Z$, so $Z$ is a homogeneous space; then $Z$ is homeomorphic to $Isom(Z)/H$ for some closed subgroup $H \subset Isom(Z)$. Now we can consider on $Z$ the restricted Haar measure $\mu_Z$ induced from the topological group with uniform topology $G = Isom(Z)$ (see \cite{R}).  \
Next let us take an arbitrary  function $\psi: X \to Isom(Z)$ and define the group extension $f_\psi: X \times Z \to X \times Z$,
$$f_\psi(x, z) = (f(x), \psi(x)(z)), (x, z) \in X \times Z$$

The function $\psi$ is called a \textit{cocycle} and $f_\psi$ a \textit{cocycle extension}. On $X \times Z$ we consider the product measure $ \mu \times \mu_Z$, where $\mu_Z$ is the induced Haar measure on $Z$. The cocycle $\psi$ is called a \textit{coboundary} with respect to the endomorphism $f: X \to X$, if there exists a measurable function $\chi$ and a constant $c$ so that $\psi = \chi \circ f - \chi + c$, $\mu$-almost everywhere.

Rudolph showed in \cite{R} that, if $(X, \mathcal{B}, \mu, f)$ is isomorphic (i.e measure-theoretically conjugate) to a 2-sided Bernoulli shift and if $f_\psi$ is weak mixing with respect to $\mu \times \mu_Z$, then $(X \times Z, \mu \times \mu_Z, f_\psi)$ is 2-sided Bernoulli as well. However this result is no longer true a priori for endomorphisms, and moreover the methods and techniques for this case are different (see for instance \cite{Pa-96} where these differences are discussed for a specific extension example).

In the sequel we shall work with a specific case, namely when the
metric space $Z$ is a torus $\mathbb T^k, k \ge 1$. We will use
the additive notation on $\mathbb T^k$. Our cocycle will be given
by a map $$\psi: \mathbb T^m \to \mathbb T^k, \psi = (\psi_1,
\ldots, \psi_k),$$ with $\psi_i: \mathbb T^m \to S^1, i = 1,
\ldots, k$. The group extension of the expanding toral
endomorphism $f_A: \mathbb T^m \to \mathbb T^m$ is the skew
product $f_{A, \psi}: \mathbb T^m \times \mathbb T^k  \to  \mathbb
T^m \times \mathbb T^k$, $$f_{A, \psi}(x, z) = (f_A(x), \psi_1(x)
+ z_1 \ (\text{mod}\ 1), \ldots, \psi_k(x) + z_k \ (\text{mod}\
1)), (x, z) \in \mathbb T^m \times \mathbb T^k$$

Clearly $f_{A, \psi}$  preserves the product measure $\mu_m \times \mu_k$ on $\mathbb T^m \times \mathbb T^k$, where $\mu_m$ and  $\mu_k$ represent the Lebesgue (Haar) measures on $\mathbb T^m$, respectively on $\mathbb T^k$.

Now let us assume that the map $\psi: \mathbb T^m \to \mathbb T^k$ used above is Holder continuous. Since $f_A$ is distance expanding, we see that the branches of inverse iterates of $f_A$ contract exponentially, hence from the Holder continuity of $\psi$ it follows that $\psi$ is a \textit{summable cocycle} (see \cite{HR}, where in our case one considers local inverse iterates of $f_A$).
 \ The next step will be to assure the \textit{weak mixing} of $f_{A, \psi}$ with respect to $\mu_m \times \mu_k$.

Generally speaking, the ergodicity of a measure-preserving endomorphism $f$ means that the associated operator $U_f$ given by composition with $f$ on the space of integrable functions on $(X, \mathcal{B}, \mu)$ (see \cite{Wa}), has only constants as eigenfunctions corresponding to the eigenvalue 1. If $f$ is ergodic and, in addition, $U_f$ has no eigenfunctions except essential constants, then $f$ is said to have \textit{continuous spectrum}. For measure-preserving transformations $f$, continuous spectrum is in fact equivalent to $f$ being weak mixing (see for instance \cite{Wa}). From this we see that there is a strong relation between weak-mixing and Livsic type equations for coboundaries.

Criteria for the weak mixing of the group extension endomorphism $f_{\psi}$ were given first for skew products with rotations (\cite{AS}), then in an abstract setting (see for instance \cite{KN}, \cite{Pa-97}, \cite{R}, etc.), and are centered on the condition that $\psi$ is not a coboundary. \

In our particular case of skew products with tori, conditions for weak mixing were given in \cite{PP}.
Let us assume that $\psi$ is Holder continuous of Holder exponent $\alpha>0$, i. e $\psi \in \mathcal{C}^\alpha(\mathbb T^m, \mathbb T^k)$. Then  we have the following result:

\begin{unthm}[Mixing Conditions for Extensions]
The above expanding map $f_{A, \psi}$ is weak mixing with respect to the product of Haar measures $\mu_m \times \mu_k$ on $\mathbb T^m \times \mathbb T^k$
 if the equation
\begin{equation}\label{livsic}
F \circ f_A(x) = c + \ell_1 \psi_1(x) + \ldots + \ell_k \psi_k(x)
+ F(x) \ \text{mod} \ 1, \text{a.e}
\end{equation}
with $F: \mathbb T^m \to \mathbb R$ measurable, $(\ell_1, \ldots, \ell_k) \in \mathbb Z^k$ and $c \in \mathbb R$, has only the trivial solution $c=0, (\ell_1, \ldots, \ell_k) = (0, \ldots, 0)$ and $F$ constant.
\end{unthm}

\begin{pro}
Let $f_A: \mathbb T^m \to \mathbb T^m$ be an expanding toral
endomorphism and $\psi: \mathbb T^m \to \mathbb T^k$ be a Holder
continuous function. Assume that, if  there exist a measurable
function $F$ on $\mathbb T^m$, a constant $c \in \mathbb R$ and a
$k$-tuple of integers $(\ell_1, \ldots, \ell_k)$ with $F \circ
f_A(x) = c + \ell_1 \psi_1(x) + \ldots + \ell_k \psi_k(x) + F(x)
(\text{mod} \ 1) \ a.e,$ then $c = 0, (\ell_1, \ldots, \ell_k) =
(0, \ldots, 0)$ and $F$ is constant (i.e the equation
(\ref{livsic}) has only the trivial solution). Then the skew
product $f_{A, \psi}$ is 1-sided Bernoulli with respect to the
product of the respective Haar measures $\mu_m \times \mu_k$ on
$\mathbb T^m \times \mathbb T^k$.
\end{pro}

\begin{proof}
We know from the Holder continuity of $\psi$ and from the uniform dilation of $f_A$ on $\mathbb T^m$ that $\psi$ generates a summable cocycle with respect to $f_A$.

Also since the only solution to equation (\ref{livsic}) is the trivial one, we obtain from the above Mixing Conditions for Extensions, that $f_{A, \psi}$ is weak mixing with respect to the product measure $\mu_m \times \mu_k$.

On the other hand, we showed in Theorem \ref{main} that the expanding toral endomorphism $f_A$ is 1-sided Bernoulli with respect to $\mu_m$.
Thus we can use Theorem 6.4 of \cite{HR}, to conclude that the extension $f_{A, \psi}$ is tree very weakly Bernoulli, hence 1-sided Bernoulli with respect to $\mu_m \times \mu_k$.

\end{proof}

\textbf{Remark 1:} Given $\psi = (\psi_1, \ldots, \psi_k) \in
\mathcal{C}^\alpha(\mathbb T^m , \mathbb T^k)$, we know from
Livsic results that:  \ there exists a measurable function $F:
\mathbb T^m \to \mathbb R$ and a constant $c$ such that $F\circ
f_A(x) = c+\ell_1 \psi_1(x) + \ldots + \ell_k\psi_k(x) + F(x)
\text{mod} \ 1$ a.e if and only if there exists $\tilde F \in
\mathcal{C}^\alpha(\mathbb T^m, \mathbb R)$ such that $\tilde
F(\cdot) = F(\cdot)$ a.e and $\tilde F \circ f_A(x) = c + \ell_1
\psi_1(x) + \ldots + \ell_k \psi_k(x) + \tilde F(x) \ \text{mod} \
1$, for \textit{all} $x \in \mathbb T^m$ (see for instance
\cite{PP}).

This last condition happens if and only if, for any periodic point
$z \in \mathbb T^m$ with $f_A^n(z) = z$, we have that $$-nc =
S_n(\ell_1\psi_1 + \ldots \ell_k \psi_k) (z) \ \text{mod} \ 1,$$
where $S_n\omega(y) := \omega(y)+ \ldots + \omega(f_A^{n-1}(y)), y
\in \mathbb T^m$, defines the $n$-th consecutive sum of
$\omega(\cdot)$, for $n \ge 1$.

\textbf{Example 2.}  The above remark will help us check weak
mixing for specific expanding maps and cocycles. For instance take
the toral endomorphism $f_A$ given by the matrix $A =
\left(\begin{array}{ll}
                                       2 & 1 \\
                                       0 & 6
  \end{array} \right)$, and the cocycle $\psi: \mathbb T^2 \to \mathbb T^2$ given in additive notation by $\psi(x_1, x_2) = (\sin 2\pi(x_1 + 3x_2), \sin 2\pi x_2)$. We obtain then the toral extension $f_{A, \psi}: \mathbb T^2 \times \mathbb T^2 \to \mathbb T^2 \times \mathbb T^2$,
$$f_{A, \psi}(x_1, x_2, y_1, y_2) = (2x_1 + x_2, 6x_2, y_1 + \sin
2\pi(x_1 + 3x_2), y_2 + \sin 2\pi x_2) \ \text{mod} \ 1$$

Let us check if condition (\ref{livsic}) is satisfied, by using
the above remark. Assume there exists $(\ell_1, \ell_2) \in
\mathbb Z^2$, a Holder continuous function $F$ and a constant $c$
such that $F \circ f_A = c + \ell_1 \psi_1 + \ell_2 \psi_2 + F
(\text{mod} \ 1)$. Then $$-nc = S_n(\ell_1\psi_1+\ell_2 \psi_2)(z)
\ \text{mod} \ 1, \text{ as long as} \ f_A^n(z) = z, n \ge 1$$ But
one of the fixed points of $f_A$ is $(0, 0)$ so if
$\ell_1\psi_1(0, 0) + \ell_2 \psi_2(0, 0) = 0$, then $c = 0$.
Consider now $(\frac 45, \frac 15)$ which is another fixed point
of $f_A$. Then we should have: $$\ell_1 \psi_1(\frac 45, \frac 15)
+ \ell_2 \psi_2 (\frac 45, \frac 15) = 0 \ \text{mod} \ 1$$ This
implies that $-\ell_1 \sin \frac{\pi}{5} + \ell_2 \sin
\frac{2\pi}{5} = 0 \ \text{mod} \ 1$; but $\sin \frac{\pi}{5} =
\frac 14 \sqrt{10-2\sqrt{5}}$ and $\sin \frac{2\pi}{5} = \frac 14
\sqrt{10+2\sqrt{5}}$, so we obtain a contradiction. Hence we
conclude that the only solution of (\ref{livsic}) is the trivial
one, meaning that $f_{A, \psi}$ is weakly mixing. By Theorem
\ref{main} we obtain then that $f_{A, \psi}$ is 1-sided Bernoulli
with respect to the Lebesgue measure on $\mathbb T^4$.

Moreover by Theorem \ref{perturb} any extension of type $g_\psi$,
with $g$ a smooth perturbation of $f_A$ and $\psi$ Holder s.t
condition (\ref{livsic}) has only the trivial solution,
 turns out to be 1-sided Bernoulli with respect to the product measure $\mu_g \times \mu_2$ on $\mathbb T^4$.

\

\textbf{Remark 2:} In fact for a given expanding toral
endomorphism $f_A$ on $\mathbb T^m$, \textbf{most cocycles}
generate weak mixing extensions. Indeed let
$\mathcal{C}^\alpha(\mathbb T^m, \mathbb T^k)$ denote the space of
Holder continuous functions of exponent $\alpha$, endowed with the
norm $$||\psi|| = |\psi|_\alpha + |\psi|_\infty,$$ where
 $|\psi|_\alpha := \mathop{\sup}\limits_{x \ne y}\frac{|\psi(x) - \psi(y)|}{|x-y|^\alpha}$ and $|\psi|_\infty$ is
 the uniform norm. Then from \cite{Pa-97} it follows that the collection of functions
 $\psi \in \mathcal{C}^\alpha(\mathbb T^m, \mathbb T^k)$ which are \textbf{not}
 coboundaries in the sense of equation (\ref{livsic}), i.e which give a weak mixing extension $f_{A, \psi}$,
 contains a \textbf{dense} $G_\delta$ set in $\mathcal{C}^\alpha(\mathbb T^m, \mathbb T^k)$.

Thus for "most" cocycles $\psi \in \mathcal{C}^\alpha(\mathbb T^m,
\mathbb T^k)$, the toral extension $f_{A, \psi}$ is 1-sided
Bernoulli with respect to the product of Haar measures $\mu_m
\times \mu_k$.

The 1-sided Bernoullicity holds also for "most" toral extensions
of smooth perturbations $g$ of $f_A$, together with their
respective product measure $\mu_g \times \mu_k$.

\

\textbf{Acknowledgements:} This work was supported by CNCSIS - UEFISCDI, project PN II - IDEI 1191/2008.

\

\

\textbf{E-mail:}  Eugen.Mihailescu\@@imar.ro

Institute of Mathematics of the Romanian Academy, P. O. Box 1-764,
RO 014700, Bucharest, Romania.

Webpage: www.imar.ro/$\sim$mihailes

\end{document}